\documentclass[10pt, twoside]{article}
\usepackage{graphicx}
\usepackage{amsmath}
\usepackage{amsfonts}
\usepackage{amssymb}
\usepackage{amsthm} 
\usepackage{bm} 
\usepackage[shortlabels]{enumitem}
\usepackage[cp1250]{inputenc}
\usepackage[a4paper, left=2.5cm, right=2.5cm, top=3.5cm, bottom=3.5cm, headsep=1.2cm]{geometry}
\usepackage{multicol}

\usepackage{caption}
\usepackage{float}

\usepackage{color}
\usepackage[dvipsnames]{xcolor}
\definecolor{sepia}{cmyk}{0, 0.83, 1, 0.70}

\usepackage[plainpages=false,pdfpagelabels,bookmarksnumbered,%
 colorlinks=true,%
 linkcolor=sepia,%
 citecolor=sepia,%
 unicode]{hyperref}

\usepackage{todonotes}

\usepackage{fancyhdr}
\newcommand\shorttitle{Outer Approximation Methods for VIs Defined over the SCFP}
\newcommand\shortauthors{A. Cegielski, A. Gibali, S. Reich and R. Zalas}

\newtheorem{theorem}{Theorem}

\newtheorem{lemma}[theorem]{Lemma}

\theoremstyle{definition}

\newtheorem{definition}[theorem]{Definition}
\newtheorem{example}[theorem]{Example}

\newtheorem{notation}[theorem]{Notation}

\newtheorem{remark}[theorem]{Remark}

\numberwithin{equation}{section} 
\numberwithin{theorem}{section}  

\makeatletter
\renewenvironment{proof}[1][\proofname]
{\par
	\pushQED{$\blacksquare$} 
	\normalfont\topsep6\p@\@plus6\p@\relax
	\trivlist
	\item[\hskip\labelsep\bfseries#1\@addpunct{.}]
	\ignorespaces}
{\popQED \endtrivlist\@endpefalse}
\makeatother

\DeclareMathOperator*{\interior}{int}

\DeclareMathOperator*{\fix}{Fix}
\DeclareMathOperator*{\id}{Id}

\DeclareMathOperator*{\Argmin}{Argmin}
\DeclareMathOperator*{\argmin}{argmin}

\DeclareMathOperator*{\argmax}{argmax}

\DeclareMathOperator*{\prox}{prox}

\newcommand{\range}[1]{\mathcal R (#1)}
\newcommand{\nullset}[1]{\mathcal N (#1)}

\begin{document}

\title{\vspace{-6em}{Outer Approximation Methods for Solving Variational Inequalities Defined over the Solution Set of a Split Convex Feasibility Problem}}

\author{Andrzej Cegielski\thanks{Faculty of Mathematics, Computer Science and Econometrics, University of Zielona G\'{o}ra, ul. Szafrana 4a, 65-516 Zielona G\'{o}ra, Poland; a.cegielski@wmie.uz.zgora.pl.}
\and Aviv Gibali\thanks{Department of Mathematics, ORT Braude College, 2161002 Karmiel, Israel; avivg@braude.ac.il.}
\and Simeon Reich\thanks{Department of Mathematics, The Technion - Israel Institute of Technology, 3200003 Haifa, Israel;\newline sreich@technion.ac.il; rzalas@campus.technion.ac.il.}
\and Rafa\l\ Zalas\footnotemark[3]}

\maketitle

\begin{abstract}
We study variational inequalities which are governed by a strongly monotone and Lipschitz continuous operator $F$ over a closed and convex set $S$. We assume that $S=C\cap A^{-1}(Q)$ is the nonempty solution set of a (multiple-set) split convex feasibility problem, where $C$ and $Q$ are both closed and convex subsets of two real Hilbert spaces $\mathcal H_1$ and $\mathcal H_2$, respectively, and the operator $A$ acting between them is linear. We consider a modification of the gradient projection method the main idea of which is to replace at each step the metric projection onto $S$ by another metric projection onto a half-space which contains $S$. We propose three variants of a method for constructing the above-mentioned half-spaces by employing the multiple-set and the split structure of the set $S$. For the split part we make use of the Landweber transform.

\smallskip
\noindent \textbf{Keywords:} $CQ$-method, split convex feasibility problem, variational inequality.

\smallskip
\noindent \textbf{AMS Subject Classification:} 47H09, 47H10, 47J20, 47J25, 65K15.
\end{abstract}

\section{Introduction}
Let $\mathcal H_1$ and $\mathcal H_2$ be two real Hilbert spaces. In this paper we consider the following \textit{variational inequality} problem (VI($F$, $S$)) governed by an $L$-Lipschitz continuous and $\alpha$-strongly monotone operator $F\colon\mathcal H_1\rightarrow\mathcal H_1$ over a nonempty, closed and convex subset $S\subseteq \mathcal H_1$: find a point $x^*\in S$ for which the inequality
\begin{equation}\label{int:VI}
\langle F x^*, z-x^*\rangle\geq 0
\end{equation}
holds true for all $z\in S$.

It is well known that the \textit{gradient projection method} \cite{Gol64}
\begin{equation}\label{int:GP}
u_0\in\mathcal H_1, \qquad u_{k+1}:=P_S(u_k-\lambda F (u_k)), \qquad k=0,1,2,\ldots,
\end{equation}
generates a sequence which converges in norm to the unique solution of VI($F$, $S$) when $\lambda\in (0, \frac{2\alpha}{L^2})$. This is due to the fact that the operator $P_S(\id -\lambda F)$ becomes a strict contraction the fixed point of which coincides with the solution of VI($F$, $S$); see \cite[Theorem 46.C]{Zeidler1985} or \cite[Theorem 5]{CZ13}.

Gibali et al. \cite{GRZ17} have proposed the framework of \emph{outer approximation methods}, where the unknown parameter $\lambda$ is replaced by a null, non-summable sequence $\{\lambda_k\}_{k=0}^\infty\subseteq[0,\infty)$ and the difficult projection onto $S$ is replaced by a sequence of simpler to evaluate metric projections onto certain half-spaces $H_k \supseteq S$. The computational cost of such methods depends to a large extent on the construction of the half-spaces $H_k$ which, in \cite{GRZ17}, were obtained by using a given sequence of cutter operators  (see Definition \ref{def:QNE} below) $T_k \colon \mathcal H_1 \to \mathcal H_1$ with $S\subseteq \fix T_k$ for each $k=0,1,2\ldots$. This method can be written in the following way:
\begin{equation}
u_0\in\mathcal H_1;\qquad u_{k+1}:=R_k(u_k-\lambda_k F (u_k)), \qquad k=0,1,2,\ldots, \label{int:uk}
\end{equation}
where
\begin{equation}
R_k:=\id+\alpha_k(P_{H_k}-\id), \qquad \alpha_k\in[\varepsilon,2-\varepsilon], \qquad \varepsilon >0, \label{int:Rk}
\end{equation}
and
\begin{equation}
H_k:=\{z\in\mathcal H_1 \colon\langle u_k-T_k (u_k), z-T_k (u_k) \rangle\leq 0\}. \label{int:Hk}
\end{equation}
Its geometrical interpretation is presented in Figure \ref{fig:int}.

\begin{figure}[htb]
\centering
\includegraphics[bb=0 0 463.25 293.43, scale=0.5]{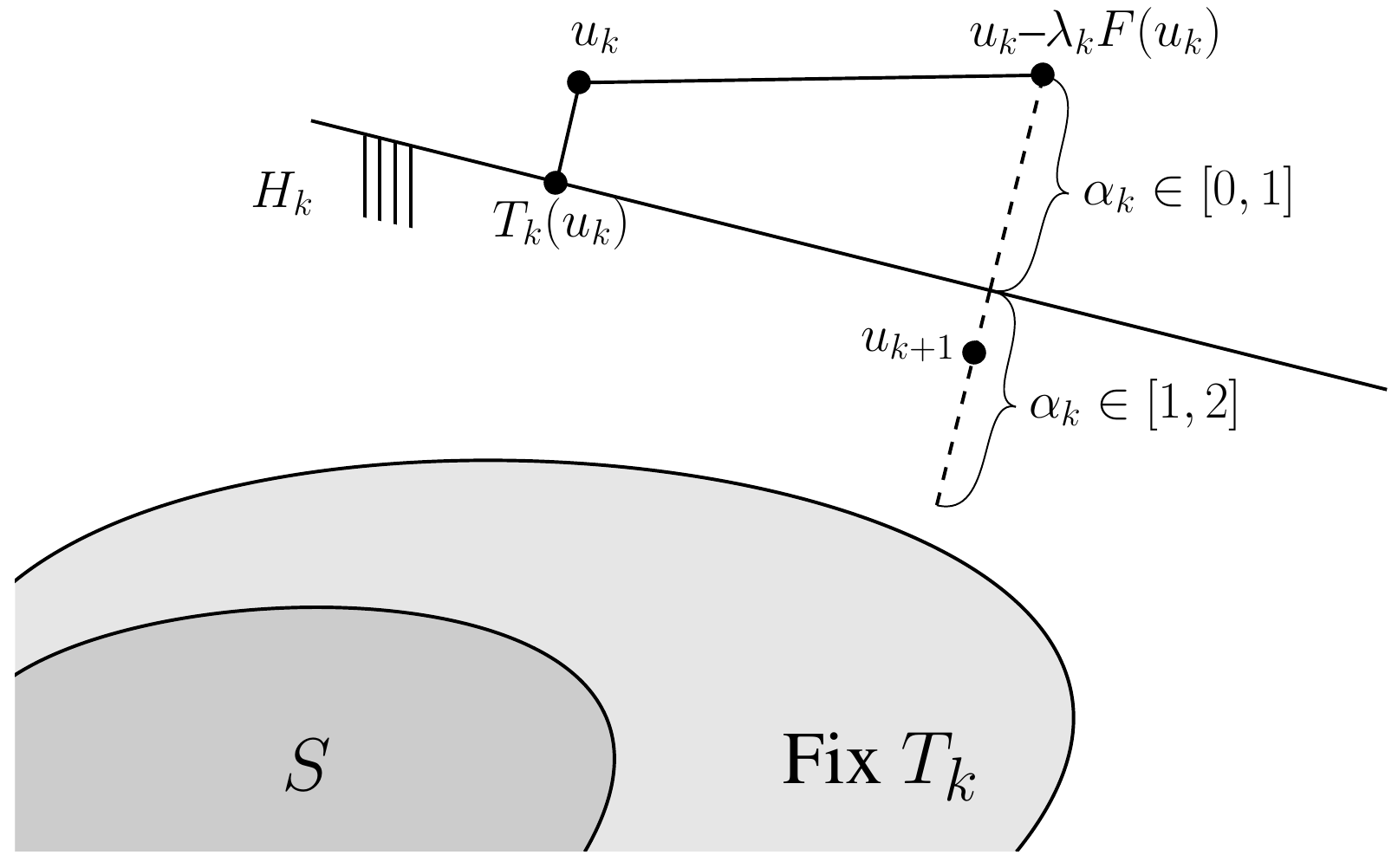}
\caption{One step of the outer approximation method \eqref{int:uk}--\eqref{int:Hk}.}
\label{fig:int}
\end{figure}

The outer approximation method has it roots in the work of Fukushima \cite{Fuk86}, where $S = \{x\in\mathbb R^d \colon s(x)\leq 0\}$ is a sublevel set of some convex function $s\colon \mathbb R^d \to \mathbb R$ and $H_k:=\{x\in\mathbb R^d \colon s(x_k)+\langle g_k, x-x_k\rangle\leq 0\}$ is a sublevel set of the linearization of $s$ at the point $x_k$ with $g_k\in\partial s(x_k)$. In this case $T_k:=P_s$ is the subgradient projection related to $s$. Other instances of this method can be found, for example, in \cite{CGRZ13, CG08, GRZ15, HY13, HT19}.

As it was already observed in \cite{GRZ17}, by a proper choice of the starting point, the outer approximation method can also be considered a particular case of the \emph{hybrid steepest descent method}
\begin{equation}\label{int:uk:HSD}
u_0\in\mathcal H;\qquad u_{k+1}:=R_k(u_k)-\lambda_k F (R_k(u_k)), \qquad k=0,1,2,\ldots,
\end{equation}
in which case $R_k\colon \mathcal H_1\to\mathcal H_1$ may be general $\rho_k$-strongly quasi-nonexpansive operators with $S\subseteq \fix R_k$ and $\inf_k \rho_k >0$. Other works related to the above method can be found, for example, in \cite{AK11, AK14, Ceg14, CM16, CZ13, CZ14, DY98, Yam01, YO04, YCL18} and even more general methods can be found in \cite{Ceg15}.

We now recall one of the main results of \cite[Theorem 3.1]{GRZ17} which concerns method \eqref{int:uk}. Note that exactly the same result holds for method \eqref{int:uk:HSD}; see \cite[Theorem 3.16]{Zal14} or \cite[Theorem 2.17]{GRZ17}.

\begin{theorem} \label{int:th:OAM}
  If $\lim_{k\to\infty} \lambda_k = 0 $, then the sequence $\{u_k\}_{k=0}^\infty$ is bounded. Moreover, if there is an integer $s\geq 1$ such that for each subsequence $\{n_k\}_{k=0}^\infty \subseteq\{k\}_{k=0}^\infty$, we have
  \begin{equation} \label{th:OAM:Regularity}%
    \lim_{k\to\infty}\sum_{l=0}^{s-1}\| T_{n_k-l} (u_{n_k-l}) - u_{n_k-l}\|=0
    \quad\Longrightarrow\quad\lim_{k\to\infty}d(u_{n_k},S)=0,
  \end{equation}
  then $\lim_{k\rightarrow\infty}d(u_k,S)=0$. If, in addition, $\sum_{k=0}^\infty \lambda_k = \infty$, then the sequence $\{u_k\}_{k=0}^\infty$ converges in norm to the unique solution of VI($F$, $S$).
\end{theorem}

In this paper we investigate the outer approximation method while assuming that $S$ is the solution set of the \emph{(multiple-set) split convex feasibility problem}, that is,
\begin{equation}\label{int:S}
S:=\left\{z \in \mathcal H_1 \mid z \in C:=\bigcap_{i\in I}C_i \quad \text{and} \quad Az\in Q:=\bigcap_{j\in J}Q_j\right\} = C \cap A^{-1}(Q),
\end{equation}
where $A\colon\mathcal H_1\rightarrow \mathcal H_2$ is a bounded linear operator and where each $C_i\subseteq \mathcal H_1$ and $Q_j\subseteq\mathcal H_2$ are closed and convex, $i\in I:=\{1,\ldots,m\}$, $j\in J:=\{1,\ldots,n\}$.

We propose a very general framework of constructing half-spaces $H_k$ that takes into account both the split and the multiple-set structure of the constraint set $S$. To this end, similarly to Theorem \ref{int:th:OAM}, we assume that we are given two sequences of strongly quasi-nonexpansive operators $\{\mathcal U_k\colon\mathcal H_1 \to \mathcal H_1\}_{k=0}^\infty$ and $\{\mathcal V_k\colon\mathcal H_2 \to \mathcal H_2\}_{k=0}^\infty$ for which $C\subseteq \fix \mathcal U_k$ and $Q\subseteq \fix \mathcal V_k$.

Examples of such operators can be obtained by simply using the metric projections $P_{C_i}$ and $P_{Q_j}$ organized in cyclic, simultaneous or block iterative ways. A similar strategy could be applied to sublevel sets, where $C_i := \{x\in \mathcal H_1 \colon c_i(x)\leq 0\}$ and $Q_j := \{y \in \mathcal H_2\colon q_j(y) \leq 0\}$ for weakly lower semicontinuous and convex functions $c_i\colon \mathcal H_1 \to \mathbb R$ and $q_j\colon\mathcal H_2 \to \mathbb R$ or, in the general fixed point setting, where $C_i:=\fix U_i$ and $Q_j:=\fix V_j$ with cutters $U_i$ and $V_j$. In the former case the metric projections should be replaced by subgradient projections $P_{c_i}$ and $P_{q_j}$ whereas in the latter case one should simply use $U_i$ and $V_j$. For more details see Example \ref{ex:UkVk} below.

The main difficulty in finding an explicit formulation for the half-spaces $H_k$, as they are defined in \eqref{int:Hk}, lies in the sets $A^{-1}(Q)$ and $A^{-1}(Q_j)$ the projections onto which are, in general, computationally expensive. We overcome this difficulty by using the so called \emph{(extrapolated) Landweber transform} (see Definitions \ref{def:Land} and \ref{def:ELand}). Roughly speaking, the Landweber transform can be considered a formalization of several techniques used for solving split feasibility problems many of which originate in the Landweber method \cite{Lan51}. In particular, it can be informally found in the well-known $CQ$-method introduced by Byrne \cite{Byr02, Byr04} and further studied in \cite{Ceg15b, CM16, CS09, Mou10, WX11, Xu10, Xu11}. The simultaneous counterparts of the $CQ$-method can be found in \cite{Ceg15b, CEKB05, CS09, MR07, Xu06}. Such a transform, when applied to an operator on $\mathcal H_2$, say $\mathcal V_k$, defines a new operator on $\mathcal H_1$, which we denote by $\mathcal L\{\mathcal V_k\}$. As it was summarized in \cite{CRZ19}, the Landweber transform preserves many of the relevant properties of its input operator. In addition, under some assumptions, which are satisfied in our case, we have $\fix \mathcal L\{\mathcal V_k\} =  A^{-1}(\fix \mathcal V_k)$, which makes it a very suitable tool for handling split problems. The extrapolated Landweber transform has it roots in \cite{LMWX12} and can also be found in \cite{Ceg16}.

Our main contribution in the present paper is to propose three approaches to define the half-spaces $H_k$, which under certain conditions, guarantee the norm convergence of the generated iterates to the unique solution of the variational inequality \eqref{int:VI} over the subset S defined by \eqref{int:S}. The first one is based on the product of the operators $\mathcal U_k$ and $\mathcal L\{\mathcal V_k\}$, which for $\mathcal U_k = P_C$ and $\mathcal V_k = P_Q$ resembles the $CQ$-method. The second one is based on averaging between $\mathcal U_k$ and $\mathcal L\{\mathcal V_k\}$, which corresponds to the simultaneous $CQ$-method, whereas the third variant relies on the alternating use of $\mathcal U_k$ and $\mathcal L\{\mathcal V_k\}$; see Theorem \ref{th:main} for more details. In our convergence analysis, we impose two conditions on the sequences $\{\mathcal U_k\}_{k=0}^\infty$ and $\{\mathcal V_k\}_{k=0}^\infty$ which, when combined with an additional bounded regularity of two families of sets, guarantee \eqref{th:OAM:Regularity}. In particular, when $S = C$ is the solution set of the convex feasibility problem, then we obtain another convergence result along the lines of Theorem \ref{int:th:OAM}; see Theorem \ref{th:main2}. Furthermore, we provide several examples of defining $\mathcal U_k$ and $\mathcal V_k$ depending on the representation of the constraint sets $C_i$ and $Q_j$; see Example \ref{ex:UkVk}.

Our paper is organized as follows. In section \ref{sec:preliminaries} we provide necessary tools to be used in our convergence analysis. In particular, we recall the closed range theorem, some basic properties of quasi-nonexpansive operators, regular operators and the Landweber transform. In section \ref{sec:main} we present our main result (Theorem \ref{th:main}) together with some examples.

\section{Preliminaries} \label{sec:preliminaries}
Let $\mathcal H,\ \mathcal H_1$ and $\mathcal H_2$ be real Hilbert spaces. We denote by $\nullset A$, $\range A$ and $\|A\|$ the null space, the range and the norm of a bounded linear operator $A\colon\mathcal H_1\rightarrow\mathcal H_2$, respectively. It is not difficult to see that
\begin{equation}\label{eq:smallNormA}
\|A\|:=\sup\{\|Ax\|\colon x\in \mathcal N(A)^\perp \text{ and } \|x\|=1\}.
\end{equation}
Analogously, we define
\begin{equation}\label{eq:smallNormA}
|A|:=\inf\{\|Ax\|\colon x\in \mathcal N(A)^\perp \text{ and } \|x\|=1\}.
\end{equation}

\begin{theorem}[Closed Range Theorem]\label{th:closedRange} Let $A\colon\mathcal H_1\rightarrow\mathcal H_2$ be a nonzero bounded linear operator. Then the following statements are equivalent:
\begin{multicols}{3}
\begin{enumerate}
    \item[$\mathrm{(i)}$] $\mathcal R(A)$ is closed;
    \item[$\mathrm{(iv)}$] $\mathcal R(A^*)$ is closed;
    \item[$\mathrm{(vii)}$] $\mathcal R(AA^*)$ is closed;
    \item[$\mathrm{(x)}$] $\mathcal R(A^*A)$ is closed;

    \item[$\mathrm{(ii)}$] $\mathcal R(A) = (\mathcal N(A^*)^\perp$;
    \item[$\mathrm{(v)}$] $\mathcal R(A^*) = \mathcal N(A)^\perp$;
    \item[$\mathrm{(viii)}$] $\mathcal R(AA^*) = \mathcal N(A A^*)^\perp$;
    \item[$\mathrm{(xi)}$] $\mathcal R(A^*A) = \mathcal N(A^* A)^\perp$;

    \item[$\mathrm{(iii)}$] $|A| >0$;
    \item[$\mathrm{(vi)}$] $|A^*| >0$;
    \item[$\mathrm{(ix)}$] $|A A^*| >0$;
    \item[$\mathrm{(xii)}$] $| A^*A|>0$.
\end{enumerate}
\end{multicols}
\noindent Moreover, we have
\begin{equation}
    |A| = |A^*| = \sqrt{|A^*A|} = \sqrt{|AA^*|}.
\end{equation}
\end{theorem}

\begin{proof}
See \cite[Lemma 3.2]{CRZ19}.
\end{proof}

\begin{remark}
Recall that the norm of $A$ satisfies $\|A\| = \|A^*\| = \sqrt{\|A^*A\|} = \sqrt{\|AA^*\|}$. Moreover, by the definition of $|A|$ and $\|A\|$, for all $x\in \nullset A^\perp$, we have
\begin{equation}
    |A|\cdot\|x\|\leq \|Ax\|\leq \|A\|\cdot\|x\|.
\end{equation}
\end{remark}

\subsection{Quasi-Nonexpansive Operators}\label{sec:QNE}

For a given $U\colon\mathcal H\rightarrow\mathcal H$ and $\alpha
\in (0,\infty)$, the operator $U_{\alpha}:=\id +\alpha(U-\id )$ is called an $\alpha$-\textit{relaxation of} $U$, where by $\id $ we denote the identity operator. We call $\alpha$ a \textit{relaxation parameter}. It is easy to see that for every such $\alpha$, $\fix U = \fix U_\alpha$, where $\fix U:=\{z\in\mathcal H\mid U(z)=z\}$ is the \textit{fixed point set} of $U$.

\begin{definition}\label{def:QNE}
Let $U\colon \mathcal H\rightarrow\mathcal H$ be an
operator with a fixed point, that is, $\fix U
\neq \emptyset$. We say that $U$ is
\begin{enumerate}[(i)]
\item \textit{quasi-nonexpansive} (QNE) if for all $x\in\mathcal H$ and all $z\in\fix U$,
\begin{equation}
\| U(x) -z\|\leq\| x-z\|;
\end{equation}

\item  $\rho$\textit{-strongly quasi-nonexpansive} ($\rho$-SQNE), where $\rho\geq 0$, if for all $x\in\mathcal H$ and all $z\in\fix U$,
\begin{equation}
\| U(x)-z\|^2\leq\| x-z\|^2-\rho\| U(x) -x\|^2;
\end{equation}

\item a \textit{cutter} if for all $x\in\mathcal H$ and all $z\in\fix U$,
\begin{equation}
\langle z-U(x) ,x-U(x) \rangle\leq 0.
\end{equation}
\end{enumerate}
\end{definition}

For a historical and mathematical overview of the above-mentioned operators we refer the reader to \cite{Ceg12}.

\begin{theorem}
\label{th:cuttersAndQNE} Let $U\colon\mathcal H\to\mathcal H$
be an operator with $\fix U\neq\emptyset$ and let $\rho\geq 0$. Then the operator $U$ is $\rho$-SQNE if and only if $\id +\frac{1+\rho}{2}(U-\id )$ is a cutter.
\end{theorem}

\begin{proof}
See, for example, \cite[Corollary 2.1.43]{Ceg12}.
\end{proof}

\begin{theorem} \label{th:SQNE:conv}
Let $U_{i}\colon\mathcal H\to\mathcal H$ be $\rho_{i}$-SQNE, where $\rho_i > 0$, $i=1,\ldots,m$. If $\bigcap_{i=1}^m\fix  U_{i} \neq\emptyset$, then the operator $U := \sum_{i=1}^m \omega_i U_i$, where $\omega_{i}>0$, $\sum_{i=1}^m\omega_{i}=1$, is $\rho$-SQNE with
\begin{equation}
  \rho := \left[\left( \sum_{i=1}^{m}\frac{\omega_i}{\rho_i+1}\right)^{-1} -1\right]
  \geq \min \rho_i >0
\end{equation}
and $\fix U = \bigcap_{i=1}^m\fix  U_{i}$. Moreover, for all $x\in\mathcal H$, we have
\begin{equation}
  2d(x,\fix U) \cdot \left\| U(x)-x\right\|\geq \sum_{i=1}^m\omega_i\rho_i\|U_i(x)-x\|^2.
\end{equation}
\end{theorem}

\begin{proof}
See \cite[Theorem 2.1.50]{Ceg12} and \cite[Proposition 4.5]{CZ14}.
\end{proof}

\begin{theorem} \label{th:SQNE:prod}
Let $U_{i}\colon\mathcal H\to\mathcal H$ be $\rho_{i}$-SQNE. If $\bigcap_{i=1}^m\fix  U_{i} \neq\emptyset$, then the operator $ U := U_{m}\ldots U_{1}$ is $\rho$-SQNE with
\begin{equation}
  \rho := \left( \sum_{i=1}^{m}\frac{1}{\rho_i}\right)^{-1}
  \geq \frac{\min \rho_i}{m} >0
\end{equation}
and $\fix U = \bigcap_{i=1}^m\fix  U_{i}$. Moreover, for all $x\in\mathcal H$, we have
\begin{equation}
  2d(x, \fix U) \cdot \| U(x)-x\|
  \geq \sum_{i=1}^m\rho_i\|Q_i(x) - Q_{i-1}(x)\|^2,
\end{equation}
where $Q_i := U_i\ldots U_1$ and $Q_0:=\id$.
\end{theorem}

\begin{proof}
See \cite[Theorems 2.1.48]{Ceg12} and \cite[Proposition 4.6]{CZ14}.
\end{proof}

\begin{example}[Subgradient Projection] \label{ex:subProj}
  Let $f\colon \mathcal{H}\to \mathbb{R}$ be a weakly lower semicontinuous and convex function with nonempty sublevel set $S:=\{x\in \mathcal{H}\colon f(x)\leq 0\}$. For each $x\in \mathcal{H}$, let  $g(x)$  be a chosen subgradient from the subdifferential set $\partial f(x):=\{g\in \mathcal{H}\colon f(y)\geq f(x)+\langle g,y-x\rangle \text{ for all }y\in \mathcal{H}\}$,  which, by \cite[Proposition 16.27]{BC17}, is nonempty. The \emph{subgradient projection} operator $P_{f}\colon \mathcal{H}\to \mathcal{H}$ is defined  by
  \begin{equation}
    P_{f}(x):=x-\frac{f(x)}{\Vert g(x)\Vert ^{2}}g(x)
  \end{equation}
  whenever $f(x)>0$ and $P_f(x):=x$, otherwise. One can show that $P_f$ is a cutter and $ \fix P_{f}=S$; see, for example, \cite[Corollary 4.2.6]{Ceg12}.
\end{example}

\begin{example}[Proximal Operator]
  Let $f\colon\mathcal H \to \mathbb R$ be a weakly lower semicontinuous and convex function. The \emph{proximal operator}, defined by
  \begin{equation}\label{}
    \prox\nolimits_f (x) := \argmin_{y\in \mathcal H} \left( f(y)+\frac 1 2 \|y-x\|^2\right),
  \end{equation}
  is firmly nonexpansive and $\fix (\prox_f)=\Argmin_{x\in\mathcal H} f(x)$; see \cite[Propositions 12.28 and 12.29]{BC17}. Thus if $f$ has at least one minimizer, then $\prox_f$ is a cutter; see \cite[Theorem 2.2.5.]{Ceg12}.
\end{example}

\subsection{Landweber Transform}

Let $A:\mathcal{H}_{1}\rightarrow \mathcal{H}_{2}$ be a nonzero bounded linear
operator, let $V \colon\mathcal{H}_{2}\to \mathcal{H}_{2}$ be an arbitrary  operator and let $\sigma \colon \mathcal H_1 \to [1,\infty)$ be a given functional.

\begin{definition}\label{def:Land}
The operator $\mathcal{L}\{V\}:\mathcal{H}_{1}\to \mathcal{H}_{1}$ defined by
\begin{equation}
\mathcal{L}\{V\}(x):= x+\frac{1}{\| A\|^{2}}A^*\big(V(Ax)-Ax\big),
 \quad x\in \mathcal H_1,
\end{equation}%
is called the \emph{Landweber operator} (corresponding to $V$). The operation $V\mapsto \mathcal{L}\{V\}$ is called the \textit{Landweber transform}.
\end{definition}

\begin{definition}\label{def:ELand}
The operator $\mathcal{L}_\sigma\{V\}:\mathcal{H}_{1}\to \mathcal{H}_{1}$ defined by
\begin{equation}
\mathcal{L}_\sigma\{V\}(x):= x+\frac{\sigma(x)}{\| A\|^{2}}A^*\big(V(Ax)-Ax\big),
 \quad x\in \mathcal H_1,
\end{equation}%
is called the \emph{extrapolated Landweber operator} (corresponding to $V$ and $\sigma$). The operation $V\mapsto \mathcal{L}_\sigma\{V\}$ is called the \textit{extrapolated Landweber transform}.
\end{definition}

\begin{remark}\label{rem:tau}
In this paper we only consider those extrapolation functionals $\sigma$ which are bounded from above by $\tau \colon\mathcal H_1\to [1,\infty)$ defined by
\begin{equation}\label{rem:tau:eq}
  \tau(x):= \left(\frac{\|A\|\cdot \|V(Ax)-Ax\|}{\|A^*(V(Ax)-Ax)\|}\right)^2
\end{equation}
whenever $V(Ax)\neq Ax$ and $\tau(x):= 1$ otherwise. Note that $\mathcal L_\tau \{V\}(x)$ does not depend on $\|A\|$.
\end{remark}

\begin{theorem}\label{th:LanweberQNE}
If $V$ is a $\rho$-SQNE operator, where $\rho \geq 0$, and $\range A\cap \fix V\neq\emptyset$, then for every extrapolation functional $\sigma$, where $1\leq \sigma \leq \tau $, the operator $\mathcal L_\sigma \{V\}$ is $\rho$-SQNE with $\fix \mathcal L_\sigma \{V\} =  A^{-1}(\fix V)$. Moreover, for all $x\in\mathcal H_1$, we have
\begin{equation}\label{th:LanweberQNE:ineq1}
  2 d (x, \fix \mathcal L_\sigma \{V\}) \cdot \|\mathcal L_\sigma \{V\}(x)-x\|
  \geq \frac{\rho+1}{\|A\|^2} \|V(Ax) - Ax\|^2
\end{equation}
and if, in addition, the set $\mathcal R(A)$ is closed, then
\begin{equation}\label{th:LanweberQNE:ineq2}
  \frac 1 {\|A\|} d(Ax, \mathcal R(A)\cap \fix V)
  \leq d (x, \fix \mathcal L_\sigma\{V\})
  \leq \frac 1 {|A|} d(Ax, \mathcal R(A)\cap \fix V).
\end{equation}
\end{theorem}

\begin{proof}
See, for example, \cite[Theorem 4.1]{CM16} for the first two statements. Inequality \eqref{th:LanweberQNE:ineq1} follows from \cite[Lemmata 4.4 and 4.6]{CRZ19} in the case $\sigma = 1$. In the general case, where we allow $\sigma(x)\geq 1$, we use the estimate $\|\mathcal L_\sigma \{V\}(x) - x\| = \sigma(x)\|\mathcal L \{V\}(x) - x\|\geq \|\mathcal L \{V\}(x) - x\|$ and the equality $\fix \mathcal L_\sigma \{V\} = \mathcal L \{V\}$. Inequality \eqref{th:LanweberQNE:ineq2} follows from \cite[Lemma 4.4]{CRZ19}.
\end{proof}

For each pair $x,x'\in\mathcal H_1$, let $\mathcal H_1(x,x') := \{z\in\mathcal H_1 \colon \langle x-x', z-x'\rangle \leq 0\}.$ Similarly, for every pair $y,y' \in \mathcal H_2$, define $\mathcal H_2(y,y'):=\{w\in\mathcal H_2 \colon \langle y-y', w-y'\rangle \leq 0\}.$ We have the following lemma.

\begin{lemma}\label{lem:halfspace}
Assume that $V$ is a cutter  and $\range A\cap \fix V\neq\emptyset$. Then for any $u \in \mathcal H_1$, the set $ H := \mathcal H_1(u, \mathcal L_\tau \{V\}(u))$ satisfies
\begin{equation}\label{lem:halfspace1}
  H  = \{z\in\mathcal H_1 \colon \langle Au - V(Au), Az - V(Au)\rangle \leq 0\}
  = A^{-1}\big(\mathcal H_2(Au, V(Au))\big).
\end{equation}
Moreover,
\begin{equation}\label{lem:halfspace1:proj}
P_H(x) = x - \frac{(\langle Au-V(Au), Ax-V(Au)\rangle)_+}{\|A^*(Au-V(Au))\|^2} A^*(Au-V(Au))
\end{equation}
whenever $Au\neq V(Au)$ and $P_H(x)=x$, otherwise.
\end{lemma}

\begin{proof}
Assume that $Au = V(Au)$ for some $u\in \mathcal H_1$. It is easy to see that in this case all the sets in \eqref{lem:halfspace1} are equal to $\mathcal H_1$ and hence $P_H(x)=x$. Assume now that $Au\neq V(Au)$. This implies, by Theorem \ref{th:LanweberQNE},  that $u \neq \mathcal L_\tau\{V\}(u)$. A direct calculation shows that
\begin{align}\label{} \nonumber
  H & = \{z\in\mathcal H_1 \colon
    \langle u - \mathcal L_{\tau}\{ V\}(u), z - \mathcal L_\tau\{V\}(u)\rangle \leq 0\}\\ \nonumber
    & = \left\{z\in\mathcal H_1 \colon
    \left\langle - \frac{\tau(u)}{\|A\|^2} A^*(V(Au)-Au),
    z - u - \frac{\tau(u)}{\|A\|^2} A^*(V(Au)-Au) \right\rangle \leq 0\right\}\\ \nonumber
  & = \{z\in\mathcal H_1 \colon
    - \langle A^*( V(Au)-Au), z - u \rangle
    +  \| V(Au)-Au\|^2 \leq 0\}\\ \nonumber
  & = \{z\in\mathcal H_1 \colon
    \langle Au - V(Au), Az - Au \rangle
    +  \|Au - V(Au)\|^2 \leq 0\}\\ \nonumber
  & = \{z\in\mathcal H_1 \colon
    \langle Au - V(Au), Az - V(Au) \rangle \leq 0\}\\ \nonumber
  & = \{z\in\mathcal H_1 \colon Az \in \mathcal H_2(Au, V(Au)) \}\\
  & = A^{-1}\big(\mathcal H_2(Au, V(Au))\big).
\end{align}
In order to show formula \eqref{lem:halfspace1:proj} it suffices to represent the half-space $H$ as $\{z\in\mathcal H_1 \colon \langle a,z\rangle \leq \beta \}$ with nonzero $a\in \mathcal H_1$ and $\beta \in \mathbb R$ for which $P_H(x) = x - \frac{(\langle a,x \rangle - \beta)_+}{\|a\|^2} a$; see \cite[Chapter 4]{Ceg12}.
\end{proof}

\begin{lemma}
Let $P_q$ be a subgradient projection for a weakly lower semicontinuous function $q\colon\mathcal H_2\to\mathbb R$ with the corresponding subgradients $h(y)\in \partial q(y) $, $y\in \mathcal H_2$, and assume that $q(Az)\leq 0$ for some $z \in \mathcal H_1$. Then for any $u \in \mathcal H_1$, the set $H := \mathcal H_1(u, \mathcal L_\tau \{P_q\}(u))$ satisfies
\begin{equation}\label{lem:halfspace2}
  H = \{z\in \mathcal H_1 \colon q(A u) + \langle A^* h(Au), z - u\rangle \leq 0\}
\end{equation}
whenever $q(Au)>0$ and $H = \mathcal H_1$, otherwise. Consequently,
\begin{equation}\label{lem:halfspace2:proj}
  P_H(x) = x - \frac{(q(Au) + \langle A^* h(Au), x - u\rangle)_+}{\|A^* h(Au)\|^2} A^* h(Au)
\end{equation}
whenever $q(Au)>0$ and $P_H(x)=x$, otherwise.
\end{lemma}

\begin{proof}
Fix a point $u\in \mathcal H_1$. Assume first that $q(Au) \leq 0$. Since $\fix P_q = \{y\in\mathcal H_2 \colon q(y)\leq 0\}$ (see Example \ref{ex:subProj}), we see that $Au=P_q(Au)$. Consequently, $u = \mathcal L_\tau\{V\}(u)$ and thus $H=\mathcal H_1$. Now assume that $q(Au)>0$ in which case $Au\neq P_q(Au)$. Let $h(Au) \in \partial q(Au)$. Then, by \eqref{lem:halfspace1} applied to $V=P_q$, we obtain
\begin{align}\label{} \nonumber
  H & = \{z\in\mathcal H_1 \colon
    \langle Au - P_{q}(Au), Az - P_{q}(Au) \rangle \leq 0\}\\ \nonumber
  & = \left\{z\in\mathcal H_1 \colon
    \left\langle \frac{q(Au)}{\|h(Au)\|^2} h, (Az - Au) + \frac{q(Au)}{\|h(Au)\|^2} h(Au) \right\rangle \leq 0\right\} \\
  & = \{z\in\mathcal H_1 \colon
    \langle A^* h(Au), z - u \rangle + q(Au) \leq 0\}.
\end{align}
Equation \eqref{lem:halfspace2:proj} follows by the same argument as in the proof of Lemma \ref{lem:halfspace}.
\end{proof}

\begin{remark}
  In view of \cite[Theorem 16.47]{BC17}, we get $A^* \partial q (A u) = \partial (q \circ A)(u)$. Consequently, the half-space $H$ defined in \eqref{lem:halfspace2} becomes a sublevel set of the functional $q\circ A$ linearized at $u$.
\end{remark}

\subsection{Regular sets}\label{sec:regularSets}
Let $C_{i}\subseteq \mathcal{H}$, $i\in I$, be closed and convex sets with a nonempty intersection $C$. Following Bauschke \cite[Definition 2.1]{Bauschke1995}, we propose the following definition.
\begin{definition}\label{def:BR}
 We say that the family $\mathcal C:=\{C_{i}\mid i\in I\}$ is \textit{boundedly regular} if for any bounded sequence $\{x_k\}_{k=0}^{\infty }\subseteq \mathcal H$, the following implication holds:
\begin{equation}
\lim_{k\rightarrow\infty}\max_{i\in I}d(x_k,C_i)=0\quad\Longrightarrow\quad \lim_{k\rightarrow\infty}d(x_k,C)=0.
\end{equation}
\end{definition}

\begin{example}
\label{th:bdreg}If at least one of the following conditions is satisfied: (i) $\dim \mathcal{H}<\infty $, (ii) $\interior \bigcap_{i\in I}C_{i}\neq \emptyset$ or (iii) each $C_{i}$ is a half-space, then the family $\mathcal{C}:=\{C_{i}\mid i\in I\}$ is boundedly
regular; see \cite{BauschkeBorwein1996}.
\end{example}

\subsection{Regular Operators}

\begin{definition}\label{def:WRBR}
We say that a quasi-nonexpansive operator $U\colon \mathcal H\rightarrow \mathcal
H$ is \textit{boundedly regular} if for any bounded sequence $\{x_{k}\}_{k=0}^\infty \subseteq \mathcal H$, we have
  \begin{equation} \label{eq:def:BR}
    \lim_{k\to\infty}\| U(x_k)-x_k\| =0\quad\Longrightarrow\quad \lim_{k\to\infty}d(x_k,\fix U)=0.
    \end{equation}
\end{definition}

\begin{notation}
We define $\prod_{j\in J}U_j:=U_{j_m}\ldots U_{j_1}$ to be the product of operators $U_i\colon\mathcal H \to \mathcal H$, $i\in I$, over a nonempty ordered index set $J=(j_1,\ldots,j_m) \subseteq I$.

\end{notation}

\begin{theorem}\label{th:regularTk}
Let $U_i\colon\mathcal H\rightarrow\mathcal H$ be boundedly regular cutters, $i\in I=\{1,\ldots,m\}$ and assume that $\bigcap_{i\in I}\fix U_i\neq\emptyset$. For each  $k=0,1,2,\ldots,$ let $I_k\subseteq I$ be a nonempty (ordered) subset, $|I_k|\leq m$ and let $0<\omega\leq \omega_{i,k}\leq 1$ be such that $\sum_{i\in I_k}\omega_{i,k}=1$. Then for every bounded sequence  $\{x_k\}_{k=0}^\infty\subseteq \mathcal H$, we have
\begin{equation}\label{eq:th:regularTk:result1}
\lim_{k\to\infty}\left\|\sum_{i\in I_k}\omega_{i,k}U_{i}(x_k)-x_k\right\|=0 \quad \Longrightarrow \quad
\lim_{k\to\infty} \max_{i\in I_k}d(x_k,\fix U_i)=0
\end{equation}
and
\begin{equation}\label{eq:th:regularTk:result2}
\lim_{k\to\infty}\left\|\prod_{i\in I_k}U_{i}(x_k)-x_k\right\|=0 \quad \Longrightarrow \quad
\lim_{k\to\infty} \max_{i\in I_k}d(x_k,\fix U_i)=0.
\end{equation}
\end{theorem}
\begin{proof}
See, either \cite[Lemma 4.11]{Zal14} or \cite[Lemma 3.5]{RZ16}.
\end{proof}

\section{Main Result}\label{sec:main}

\begin{theorem}\label{th:main}
Let $F\colon\mathcal H_1\rightarrow\mathcal H_1$ be $L$-Lipschitz continuous and $\alpha$-strongly monotone, and let $S\subseteq\mathcal H_1$ be the nonempty solution set of the split convex feasibility problem, that is,
\begin{equation}\label{th:main:S}
S:=\left\{z \in \mathcal H_1 \mid z \in C:=\bigcap_{i\in I}C_i \quad \text{and} \quad Az\in Q:=\bigcap_{i\in J} Q_j\right\},
\end{equation}
where each $C_i \subseteq \mathcal H_1$, $Q_j \subseteq \mathcal H_2$ are closed and convex, $i\in I:=\{1,\ldots,m\}$, $j\in J:=\{1,\ldots,n\}$ and where $A\colon\mathcal H_1\rightarrow\mathcal H_2$ is a bounded linear operator. Moreover, for each $k=0,1,2,\ldots$, let $\mathcal U_k\colon\mathcal H_1 \to \mathcal H_1$ be $\beta_k$-SQNE with $C\subseteq \fix \mathcal U_k$ and $\beta:=\inf \beta_k >0$, and let $\mathcal V_k\colon\mathcal H_2\rightarrow\mathcal H_2$ be $\gamma_k$-SQNE with $Q\subseteq \fix \mathcal V_k$ and $\gamma := \inf_k \gamma_k > 0$. Furthermore, for each $k=0,1,2,\ldots$, let $\sigma_k \colon \mathcal H_1 \to [1,\infty)$ be an extrapolation functional bounded from above by $\tau_k \colon \mathcal H_1 \to [1,\infty)$ defined by
\begin{equation}\label{th:main:tauk}
  \tau_k(x):= \left(\frac{\|A\|\cdot \|\mathcal V_k(Ax)-Ax\|}{\|A^*(\mathcal V_k(Ax)-Ax)\|}\right)^2
\end{equation}
whenever $\mathcal V_k(Ax) \neq Ax$ and $\tau_k(x):=1$, otherwise. Let the sequence $\{u_k\}_{k=0}^\infty$ be defined by the outer approximation method \eqref{int:uk}--\eqref{int:Hk} combined with one of the following algorithmic operators $T_k$:
\begin{enumerate}[(i)]
  \item product operators, where
  \begin{equation} \label{th:main2:Tk:prod}
    T_k (x):= x + \frac{1+\rho_k}{2} \Big(\mathcal U_k (\mathcal L_{\sigma_k} \{\mathcal V_k\}(x))  - x \Big),
  \end{equation}
  \begin{equation} \label{th:main2:rhok:prod}
    0 \leq \rho_k \leq
    \left(\frac{1}{\beta_k} + \frac{1}{\gamma_k}\right)^{-1};
  \end{equation}

  \item simultaneous operators, where $0 < \eta \leq \eta_k \leq 1-\eta$,
  \begin{equation}\label{th:main2:Tk:sim}
    T_k (x):= x + \frac{1+\rho_k}{2}\big(\eta_k \mathcal U_k(x) + (1-\eta_k)\mathcal L_{\sigma_k} \{\mathcal V_k\}(x) - x\big),
  \end{equation}
  \begin{equation}\label{th:main2:rhok:sim}
    0 \leq \rho_k \leq
    \left( \frac{\eta_k}{\beta_k+1} + \frac{1-\eta_k}{\gamma_k+1}\right)^{-1}-1;
  \end{equation}

  \item alternating operators, where
  \begin{equation}\label{th:main2:Tk:alt}
    T_{2k}(x) := x + \frac{1+\beta_k}{2}\big(\mathcal U_k(x) - x\big)
    \qquad \text{and} \qquad
    T_{2k+1}(x) := x + \frac{1+\gamma_k}{2}\big(\mathcal L_{\sigma_k}\{\mathcal V_k\}(x) - x\big).
  \end{equation}
\end{enumerate}
Assume that for all bounded sequences $\{x_k\}_{k=0}^\infty\subseteq \mathcal H_1$ and $\{y_k\}_{k=0}^\infty\subseteq \mathcal H_2$, we have
\begin{equation}\label{th:main:Uk}
\lim_{k\to\infty}\|\mathcal U_k(x_k)-x_k\|=0 \quad \Longrightarrow \quad
\lim_{k\to\infty} \max_{i\in I_k}d(x_k,C_i)=0,
\end{equation}
and
\begin{equation}\label{th:main:Vk}
\lim_{k\to\infty}\|\mathcal V_k(y_k)-y_k\|=0 \quad \Longrightarrow \quad
\lim_{k\to\infty} \max_{j\in J_k}d(y_k,Q_j)=0,
\end{equation}
where $I_k\subseteq I$ and $J_k\subseteq J$ are not empty and $|I_k|\leq m$, $|J_k|\leq n$.
If $\{I_k\}_{k=0}^\infty$ and $\{J_k\}_{k=0}^\infty$ are $s$-intermittent for some $s\geq 1$ (that is, $I=I_k\cup\ldots\cup I_{k+s-1}$, $J=J_k\cup\ldots\cup J_{k+s-1}$ for all $k\geq 0$), $\mathcal R(A)$ is closed, $\{A^{-1}(Q), C_1, \ldots, C_m\}$ and $\{\mathcal R(A), Q_1,\ldots, Q_n\}$ are boundedly regular, and $\lim_{k \to \infty}\lambda_k=0$, then $\lim_{k\to\infty} d(u_k,S)=0$.
If, in addition, $\sum_{k=0}^\infty\lambda_k=\infty$, then the sequence $\{u_k\}_{k=0}^\infty$ converges in norm to the unique solution of VI($F$, $S$).
\end{theorem}

\begin{proof}
Observe that the operators $T_k$ defined either in (i), (ii) or (iii) are cutters such that $S\subseteq \fix T_k$. This follows from Theorems \ref{th:cuttersAndQNE}, \ref{th:SQNE:conv}, \ref{th:SQNE:prod} and \ref{th:LanweberQNE}. Therefore it is reasonable to consider the outer approximation method paired with the $T_k$'s.

In order to complete the proof, in view of Theorem \ref{int:th:OAM}, it suffices to show that for any subsequence $\{n_k\}_{k=0}^\infty \subseteq\{k\}_{k=0}^\infty$, we have
\begin{equation} \label{pr:main:toShow}
\lim_{k \to \infty}\sum_{l=0}^{2s-1}\| T_{n_k-l} (u_{n_k-l}) - u_{n_k-l}\| = 0
\quad\Longrightarrow\quad
\lim_{k \to \infty} d(u_{n_k},S)=0.
\end{equation}
To this end, assume that
\begin{equation} \label{pr:main:assumption}
\lim_{k \to \infty}\sum_{l=0}^{2s-1}\| T_{n_k-l} (u_{n_k-l}) - u_{n_k-l}\| = 0
\end{equation}
for some $\{n_k\}_{k=0}^\infty \subseteq\{k\}_{k=0}^\infty$. We divide the rest of the proof into several steps.

\textbf{Step 1.} By Theorem \ref{int:th:OAM}, the sequence $\{u_k\}_{k=0}^\infty$ is bounded and hence $d(u_k, S)\leq R$ for some $R>0$. Moreover, by \cite[Lemma 3.2]{GRZ17}, for any subsequence $\{m_k\}_{k=0}^\infty \subseteq\{k\}_{k=0}^\infty$, we have
\begin{equation} \label{pr:main:equivalence}
  \lim_{k\to\infty}(T_{m_k} (u_{m_k}) - u_{m_k})=0 \quad \Longleftrightarrow \quad  \lim_{k\to\infty} (u_{m_k+1}-u_{m_k})=0.
\end{equation}
Consequently, by setting $m_k:=n_k-l$ and by \eqref{pr:main:assumption}, for each $l=1,2,\ldots, 2s-1$, we obtain
\begin{equation}\label{pr:main:increase}
\lim_{k\to\infty}\|u_{n_k} - u_{n_k - l}\| = 0.
\end{equation}

\textbf{Step 2.} Observe that property \eqref{th:main:Uk}, which is solely related to the sequence of operators paired with the sequence of index sets, is hereditary with respect to any of their subsequences. To be more precise, for all bounded sequences $\{x_k\}_{k=0}^\infty\subseteq\mathcal H_1$ and for any subsequence $\{m_k\}_{k=0}^\infty \subseteq\{k\}_{k=0}^\infty$, we have
\begin{equation}\label{pr:main:Umk_xk}
  \lim_{k\to\infty}\| \mathcal U_{m_k}(x_k)-x_k\|=0 \quad \Longrightarrow \quad
  \lim_{k\to\infty} \max_{i\in I_{m_k}}d(x_k,C_i)=0.
\end{equation}
Indeed, take any $z\in S$ and define $x_m':= x_{m_k}$ whenever $m = m_k$ and otherwise set $x_m' := z$. It is not difficult to see that the augmented sequence $\{x_m'\}_{n=0}^\infty$ is bounded and satisfies \eqref{th:main:Uk} which in turn implies \eqref{pr:main:Umk_xk}.

By applying a similar argument to property \eqref{th:main:Vk}, we obtain that for all bounded sequences $\{y_k\}_{k=0}^\infty\subseteq\mathcal H_2$ and for any subsequence $\{m_k\}_{k=0}^\infty \subseteq\{k\}_{k=0}^\infty$,
\begin{equation}\label{pr:main:Vmk_yk}
  \lim_{k\to\infty}\| \mathcal V_{m_k}(y_k)-y_k\|=0 \quad \Longrightarrow \quad
  \lim_{k\to\infty} \max_{j\in J_{m_k}}d(y_k, Q_j)=0.
\end{equation}

\textbf{Step 3.} We show that in all three cases (i)--(iii), we have
\begin{equation} \label{pr:main:maxDistIJ}
\lim_{k \to \infty} \max_{i\in I} d(u_{n_k}, C_i) = 0 \qquad \text{and} \qquad
\lim_{k \to \infty} \max_{j\in J} d(A u_{n_k}, Q_j) = 0.
\end{equation}
To this end, let $i_k:= \argmax_{i\in I} d(u_{n_k}, C_i)$ and let $j_k:=\argmax_{j\in J} d(Au_{n_k},Q_j)$.

\textbf{Case (i)}. By Theorems \ref{th:SQNE:prod} and \ref{th:LanweberQNE}, for each $l=0,1,2,\ldots,2s-1$, we have
\begin{align}\label{pr:main:estimate:prod}  \nonumber
  \|T_{n_k-l}(u_{n_k-l}) - u_{n_k-l}\|
  & \geq  \frac {1}{2} \| \mathcal U_{n_k-l} (\mathcal L_{\sigma_{n_k-l}}\{\mathcal V_{n_k-l}\} (u_{n_k-l})) - u_{n_k-l}\|\\ \nonumber
  & \geq \frac {\beta} {4R} \| \mathcal U_{n_k-l} (\mathcal L_{\sigma_{n_k-l}}\{\mathcal V_{n_k-l}\}(u_{n_k-l})) - \mathcal L_{\sigma_{n_k-l}}\{\mathcal V_{n_k-l}\}(u_{n_k-l}) \|^2 \\ \nonumber
  & \quad + \frac {\gamma} {4R}\| \mathcal L_{\sigma_{n_k-l}}\{\mathcal V_{n_k-l}\}(u_{n_k-l}) - u_{n_k-l} \|^2 \\ \nonumber
  & \geq \frac {\beta} {4R} \| \mathcal U_{n_k-l} (\mathcal L_{\sigma_{n_k-l}}\{\mathcal V_{n_k-l}\}(u_{n_k-l})) - \mathcal L_{\sigma_{n_k-l}}\{\mathcal V_{n_k-l}\}(u_{n_k-l}) \|^2 \\
  & \quad + \frac {\gamma} {16R^3 \|A\|^4} \cdot \| \mathcal V_{n_k-l} (Au_{n_k-l}) - Au_{n_k-l} \|^4.
\end{align}

For each $k\geq 2s-1$, let $l_k$ be the smallest $l\in \{0,\ldots,2s-1\}$ such that $i_k \in I_{n_k-l}$. Since the control sequence $\{I_k\}_{k=0}^\infty$ is $s$-intermittent, such an $l_k$ exists. By \eqref{pr:main:assumption}, \eqref{pr:main:estimate:prod} and \eqref{pr:main:Umk_xk} applied to $m_k:=n_k-l_k$ and $x_k:=\mathcal L_{\sigma_{n_k-l_k}} \{\mathcal V_{n_k-l_k}\}(u_{n_k-l_k})$, we obtain
\begin{equation}\label{pr:main:maxDistI:xk}
  \lim_{k\to\infty} \max_{i\in I_{n_k-l_k}} d(x_k, C_i) = 0.
\end{equation}
Moreover, by \eqref{pr:main:assumption} and \eqref{pr:main:estimate:prod}, we have
\begin{equation}\label{}
  \lim_{k\to\infty} \|x_k - u_{n_k-l_k}\| = 0
\end{equation}
and consequently,
\begin{align} \label{pr:main:maxDistI}\nonumber
\max_{i\in I}d(u_{n_k}, C_i)
& = \|P_{C_{i_k}}(u_{n_k}) - u_{n_k}\|
\leq \|P_{C_{i_k}}(x_k) - u_{n_k}\| \\ \nonumber
& \leq \|P_{C_{i_k}}(x_k) - x_k\| + \|x_k - u_{n_k-l_k}\| + \|u_{n_k-l_k} - u_{n_k}\|\\
& \leq \max_{i\in I_{n_k-l_k}}d(x_k, C_i) + \|x_k - u_{n_k-l_k}\| + \|u_{n_k-l_k} - u_{n_k}\| \to 0
\end{align}
as $k \to \infty$, which proves the first part of \eqref{pr:main:maxDistIJ}.

Similarly, for each $k\geq 2s-1$, let $r_k$ be the smallest $r \in \{0,\ldots,2s-1\}$ such that $j_k \in I_{n_k-r}$. By \eqref{pr:main:assumption}, \eqref{pr:main:estimate:prod} and \eqref{pr:main:Vmk_yk} applied to $m_k:=n_k-r_k$ and $y_k := Au_{n_k-r_k}$, we obtain
\begin{equation}\label{pr:main:maxDistI:yk}
  \lim_{k\to\infty} \max_{j\in J_{n_k-r_k}} d(y_k, Q_j) = 0.
\end{equation}
By the definition of the metric projection and by the triangle inequality, we have
\begin{align} \label{pr:main:maxDistJ}\nonumber
\max_{j\in J}d(Au_{n_k}, Q_j)
& = \|P_{Q_{j_k}}(A u_{n_k}) - A u_{n_k}\|
\leq \|P_{Q_{j_k}}(y_k) - A u_{n_k}\| \\ \nonumber
& \leq  \|P_{Q_{j_k}}(y_k) - y_k\| +  \|y_k - A u_{n_k}\|\\
& \leq \max_{j\in J_{n_k-r_k}}d(y_k, Q_j) + \|A\| \cdot \|u_{n_k} - u_{n_k-r_k}\| \to 0
\end{align}
as $k\to \infty$. This proves the second part of \eqref{pr:main:maxDistIJ}.

\textbf{Case (ii).} By Theorems \ref{th:SQNE:conv} and \ref{th:LanweberQNE}, for each $l=0,1,2,\ldots,2s-1$, we have
\begin{align}\label{pr:main:estimate:sim}  \nonumber
  \|T_{n_k-l}(u_{n_k-l}) - u_{n_k-l}\|
  & \geq  \frac {1}{2} \| \eta_{n_k-l} \mathcal U_{n_k-l} (u_{n_k-l}) + (1-\eta_{n_k-l})\mathcal L_{\sigma_{n_k-l}}\{\mathcal V_{n_k-l}\} (u_{n_k-l})\|\\ \nonumber
  & \geq \frac {\eta \beta} {4R} \| \mathcal U_{n_k-l} (u_{n_k-l}) - u_{n_k-l}) \|^2  \\ \nonumber
  & \quad + \frac {\eta\gamma} {4R}\| \mathcal L_{\sigma_{n_k-l}}\{V_{n_k-l}\}(u_{n_k-l}) - u_{n_k-l} \|^2 \\ \nonumber
  & \geq \frac {\eta \beta} {4R} \| \mathcal U_{n_k-l} (u_{n_k-l}) - u_{n_k-l}) \|^2 \\
  &  \quad + \frac {\eta\gamma} {16R^3 \|A\|^4} \cdot \frac{}{}\| \mathcal V_{n_k-l} ( A (u_{n_k-l})) - A (u_{n_k-l}) \|^4.
\end{align}
Similarly to Case (i), we can apply \eqref{pr:main:Umk_xk} to $m_k:=n_k-l_k$ and $x_k:=u_{n_k-l_k}$ in order to obtain \eqref{pr:main:maxDistI:xk} and \eqref{pr:main:maxDistI}. Moreover, by applying \eqref{pr:main:Vmk_yk} to $m_k:=n_k-r_k$ and $y_k := Au_{n_k-r_k}$, we obtain \eqref{pr:main:maxDistI:yk} and \eqref{pr:main:maxDistJ}.

\textbf{Case (iii).} We split the sequence $\{n_k\}_{k=0}^\infty$ into two disjoint subsequences consisting of all odd and all even integers, respectively. To this end, consider the quotients $q_k:=\lfloor n_k/2\rfloor$, and define the sets $K_1:=\{k \colon n_k = 2q_k +1\}$ and $K_2:=\{k \colon n_k = 2q_k \}$. Without any loss of generality, we may assume that both $K_1$ and $K_2$ are infinite. Otherwise the argument simplifies to only one of them.

Assume for now that $k\in K_1$. By using the equality $n_k-2l-1 = 2(q_k-l)$ and by the definition of $T_k$, we get
\begin{equation}\label{pr:main:estimate:np1}
  \|T_{n_k-2l-1}(u_{n_k-2l-1}) - u_{n_k-2l-1}\|
  \geq \frac{1}{2}\|\mathcal U_{q_k-l}(u_{n_k-2l-1}) - u_{n_k-2l-1}\|.
\end{equation}
Similarly, using the equality $n_k-2l = 2(q_k-l)+1$, the definition of $T_k$ and \eqref{th:LanweberQNE:ineq1}, we obtain
\begin{align}\label{pr:main:estimate:np2} \nonumber
  \|T_{n_k-2l}(u_{n_k-2l}) - u_{n_k-2l}\|
  & \geq \frac{1}{2}\|\mathcal L_{\sigma_{q_k-l}} \{\mathcal V_{q_k-l}\}(u_{n_k-2l}) - u_{n_k-2l}\|\\
  & \geq \frac{1}{4R\|A\|^2}\|\mathcal V_{q_k-l}(Au_{n_k-2l})-Au_{n_k-2l}\|^2.
\end{align}
Since both controls $\{I_k\}_{k=0}^\infty$ and $\{J_k\}_{k=0}^\infty$ are $s$-intermittent, for each $k\in K_1$, there are $l_k, r_k\in\{0,\ldots,s-1\}$ such that $i_k \in I_{q_k-l_k}$ and $j_k \in J_{q_k-r_k}$.
By \eqref{pr:main:assumption}, \eqref{pr:main:estimate:np1} and \eqref{pr:main:Umk_xk} applied to $m_k:=q_k-l_k$ and $x_k:= u_{n_k-2l_k-1}$, we obtain
\begin{equation}\label{pr:main:maxDistI:xp}
  \lim_{\substack{k\to\infty \\ k\in K_1}} \ \max_{i\in I_{q_k-l_k}} d(x_k, C_i) = 0.
\end{equation}
Moreover (compare with \eqref{pr:main:maxDistI}), we have
\begin{align} \label{pr:main:maxDistI:K1}
\max_{i\in I}d(u_{n_k}, C_i)
& \leq \max_{i\in I_{q_k-l_k}}d(x_k, C_i) + \|u_{n_k} - u_{n_k-2l_k-1}\| \to 0
\end{align}
as $k \to \infty$, $k\in K_1$. On the other hand, by \eqref{pr:main:assumption}, \eqref{pr:main:estimate:np2} and \eqref{pr:main:Vmk_yk} applied to $m_k:=q_k-r_k$ and $y_k:= Au_{n_k-2r_k}$, we obtain
\begin{equation}\label{}
  \lim_{\substack{k\to\infty\\ k\in K_1}} \ \max_{j\in J_{q_k-r_k}} d(y_k, Q_j) = 0.
\end{equation}
Moreover (compare with \eqref{pr:main:maxDistJ}), we have
\begin{align} \label{pr:main:maxDistJ:K1}
\max_{j\in J}d(Au_{n_k}, Q_j)
& \leq \max_{j\in J_{q_k-r_k}}d(y_k, Q_j) + \|A\| \cdot \|u_{n_k} - u_{n_k-2r_k}\| \to 0
\end{align}
as $k \to \infty$, $k\in K_1$. This proves \eqref{pr:main:maxDistIJ} when $k\in K_1$.

A very similar argument can be used to show that
\begin{equation} \label{}
\lim_{\substack{k \to \infty\\ k\in K_2}} \max_{i\in I} d(u_{n_k}, C_i) = 0 \qquad \text{and} \qquad
\lim_{\substack{k \to \infty\\ k\in K_2}} \max_{j\in J} d(A u_{n_k}, Q_j) = 0.
\end{equation}
This, when combined with \eqref{pr:main:maxDistI:K1} and \eqref{pr:main:maxDistJ:K1}, completes the proof of Case (iii).

\textbf{Step 4.} We show that in all three cases, we have $d(u_{n_k},S)\to 0$. Indeed, by Theorem \ref{th:LanweberQNE} (with $V=P_Q$), we have
\begin{equation}\label{}
  d(u_{n_k},A^{-1}(Q)) \leq \frac 1 {|A|} d(Au_{n_k},\mathcal R(A)\cap Q).
\end{equation}
Since $Au_{n_k}\in \mathcal R(A)$, by the second part in \eqref{pr:main:maxDistIJ} and, by the assumed bounded regularity of the family $\{\mathcal R(A),Q_1,\ldots,Q_n\}$, we obtain $d(Au_{n_k},\mathcal R(A)\cap Q)\to 0$ and thus $d(u_{n_k},A^{-1}(Q)) \to 0$. This, when combined with the first part of \eqref{pr:main:maxDistIJ} and the assumed bounded regularity of the family $\{A^{-1}(Q), C_1,\ldots,C_m\}$, lead to $d(u_{n_k},S)\to 0$, which completes the proof.
\end{proof}

\begin{remark}[$CQ$-methods]
  Assume that $\mathcal U_k := P_Q$ and $\mathcal V_k := P_Q$ for each $k=0,1,2,\ldots$. Then, within the framework of Theorem \ref{th:main}, the half-spaces $H_k$ are obtained by using the algorithmic operators $T_k$ corresponding to the (extrapolated) $CQ$-method in case (i), the simultaneous(-extrapolated) $CQ$-method in case (ii) and the alternating(-extrapolated) $CQ$-method in case (iii).
\end{remark}

We now present several examples of sequences $\{\mathcal U_k\}_{k=0}^\infty$ and $\{\mathcal V_k\}_{k=0}^\infty$ all of which satisfy conditions \eqref{th:main:Uk} and \eqref{th:main:Vk}, respectively. For this reason, assume that for each $i\in I$ and $j\in J$, we have
\begin{equation}
C_i = \fix U_i \qquad \text{and} \qquad Q_j=\fix V_j,
\end{equation}
where $U_i\colon \mathcal H_1 \to \mathcal H_1$ and $V_j \colon \mathcal H_2 \to \mathcal H_2$ are boundedly regular cutters.

\begin{remark}
Within the above setting, one can use:
\begin{itemize}
  \item Metric projections $U_i = P_{C_i}$ and $V_j=P_{Q_j}$.
  \item Subgradient projections $U_i=P_{c_i}$ and $V_j=P_{q_j}$, when $C_i=\{x\in \mathbb R^{d_1} \colon c_i(x)\leq 0\}$ and $Q_j=\{y\in \mathbb R^{d_2} \colon q_j(y)\leq 0\}$ for some convex functions $c_i\colon \mathbb R^{d_1} \to \mathbb R$ and $q_j\colon \mathbb R^{d_2} \to \mathbb R$.
  \item Proximal operators $U_i = \prox_{c_i}$ and $V_j = \prox_{q_j}$, when $C_i = \Argmin_{x} c_i(x)$ and $Q_j = \Argmin_y q_j(y)$ for $c_i$ and $q_j$ as above.
  \item Any firmly nonexpansive mappings $U_i\colon \mathbb R^{d_1} \to \mathbb R^{d_1}$ and $V_j\colon \mathbb R^{d_2} \to \mathbb R^{d_2}$.
\end{itemize}
\end{remark}

\begin{remark}
  Bounded regularity of the families $\{A^{-1}(Q), C_1, \ldots, C_m\}$ and $\{\mathcal R(A), Q_1,\ldots, Q_n\}$ holds when, for example, $\mathcal H_1=\mathbb R^{d_1}$, $\mathcal H_2 = \mathbb R^{d_2}$.
\end{remark}

\begin{example}\label{ex:UkVk}
In view of Theorem \ref{th:regularTk}, the operators $T_k$ (and thus the half-spaces $H_k$) presented in Theorem \ref{th:main} (cases (i), (ii) and (iii)) can be obtained by using:
\begin{enumerate}[(a)]
\item Sequential cutters, where
\begin{equation}
\mathcal U_k := U_{i_k} \qquad \text{and} \qquad \mathcal V_k := V_{j_k},
\end{equation}
and $\{i_k\}_{k=0}^\infty \subseteq I$ and $\{j_k\}_{k=0}^\infty \subseteq J$ are two $s$-almost cyclic control sequences, that is, $I=\{i_k,\ldots,i_{k+s-1}\}$ and $J=\{j_k,\ldots,j_{k+s-1}\}$ for all $k\geq 0$.

\item Simultaneous cutters, where
  \begin{equation}\label{}
    \mathcal U_k := \sum_{i\in I_k} \omega_{i,k} U_i \qquad \text{and} \qquad
    \mathcal V_k := \sum_{j\in J_k} \omega_{j,k}' V_j,
  \end{equation}
  $\{I_k\}_{k=0}^\infty \subseteq I$ and $\{J_k\}_{k=0}^\infty \subseteq J$ are $s$-intermittent control sequences, and where $\omega_{i,k}, \omega_{j,k}'\geq \omega>0$ are such that $\sum_{i\in I_k} \omega_{i,k} = \sum_{j\in J_k} \omega_{j,k}' = 1$.

\item Products of cutters, where
  \begin{equation}\label{}
    \mathcal U_k := \prod_{i\in I_k} U_i \qquad \text{and} \qquad
    \mathcal V_k := \prod_{j\in J_k} V_j,
  \end{equation}
  and $\{I_k\}_{k=0}^\infty \subseteq I$ and $\{J_k\}_{k=0}^\infty \subseteq J$ are as above.
\end{enumerate}
\end{example}

\begin{remark}
Observe that in the case of alternating operators (case (iii)) with the extrapolation functional $\sigma_k = \tau_k$, in view of Lemma \ref{lem:halfspace}, the half-space $H_{2k+1}$ and the associated projection $P_{H_{2k+1}}$ have equivalent forms, that is,
\begin{equation}
H_{2k+1} =
\{z \in \mathcal H_1 \colon \langle Au_{2k+1} - \mathcal V_{k}(Au_{2k+1}), Az - \mathcal V_{k}(Au_{2k+1})\rangle \leq 0\}
\end{equation}
and
\begin{equation}\label{rem:halfspace1:proj}
    P_{H_{2k+1}}(x) = x - \frac{(\langle Au_{2k+1} - \mathcal V_k(Au_{2k+1}), Ax - \mathcal V_k(Au_{2k+1})\rangle)_+}{\|A^*(Au_{2k+1}-\mathcal V_k(Au_{2k+1}))\|^2} A^*(Au_{2k+1}-\mathcal V_k(Au_{2k+1}))
  \end{equation}
whenever $Au_{2k+1} \neq \mathcal V_k(Au_{2k+1})$ and otherwise $H_{2k+1}=\mathcal H_1$, in which case $P_{H_{2k+1}}(x) = x$.
\end{remark}

By slightly adjusting the proof of Theorem \ref{th:main}, one can also obtain the following result.

\begin{theorem}\label{th:main2}
Let $F\colon\mathcal H\rightarrow\mathcal H$ be $L$-Lipschitz continuous and $\alpha$-strongly monotone, and let $S\subseteq\mathcal H$ be the nonempty solution set of a convex feasibility problem, that is, $S:=\bigcap_{i\in I}C_i$, where $C_i \subseteq \mathcal H$ are closed and convex, $i\in I:=\{1,\ldots,m\}$. Moreover, for each $k=0,1,2,\ldots$, let $T_k\colon\mathcal H \to \mathcal H$ be a cutter. Let the sequence $\{u_k\}_{k=0}^\infty$ be defined by the outer approximation method \eqref{int:uk}--\eqref{int:Hk}. Assume that for all bounded sequences $\{x_k\}_{k=0}^\infty\subseteq \mathcal H$, we have
\begin{equation}\label{th:main2:Tk}
\lim_{k\to\infty}\|T_k(x_k)-x_k\|=0 \quad \Longrightarrow \quad
\lim_{k\to\infty} \max_{i\in I_k}d(x_k,C_i)=0,
\end{equation}
where $I_k\subseteq I$ are not empty and $|I_k|\leq m$. If $\{I_k\}_{k=0}^\infty$ is $s$-intermittent, $\{C_1, \ldots, C_m\}$ is boundedly regular and $\lim_{k \to \infty}\lambda_k=0$, then $\lim_{k\to\infty} d(u_k,S)=0$. If, in addition, $\sum_{k=0}^\infty\lambda_k=\infty$, then the sequence $\{u_k\}_{k=0}^\infty$ converges in norm to the unique solution of VI($F$, $S$).
\end{theorem}

\bigskip
\textbf{Funding.} This work was partially supported by the Israel Science Foundation (Grants 389/12 and 820/17), the Fund for the Promotion of Research at the Technion and by the Technion General Research Fund.

\small

\bibliographystyle{spmpsci}
\bibliography{references}

\end{document}